\documentclass[11pt,reqno,dvipsnames]{amsart}

\usepackage{a4wide}
\usepackage{amsmath,amssymb}
\usepackage{bbm}
\usepackage[dvipsnames]{xcolor}
\usepackage[colorlinks=true, pdfstartview=FitV, linkcolor=blue, 
citecolor=Green]{hyperref}

\usepackage{scalerel,stackengine}
\stackMath
\newcommand\reallywidehat[1]{%
\savestack{\tmpbox}{\stretchto{%
  \scaleto{%
    \scalerel*[\widthof{\ensuremath{#1}}]{\kern-.6pt\bigwedge\kern-.6pt}%
    {\rule[-\textheight/2]{1ex}{\textheight}}%WIDTH-LIMITED BIG WEDGE
  }{\textheight}%
}{0.5ex}}%
\stackon[1pt]{#1}{\tmpbox}%
}

\newcommand\reallywidecheck[1]{%
\savestack{\tmpbox}{\stretchto{%
  \scaleto{%
    \scalerel*[\widthof{\ensuremath{#1}}]{\kern-.6pt\bigwedge\kern-.6pt}%
    {\rule[-\textheight/2]{1ex}{\textheight}}%WIDTH-LIMITED BIG WEDGE
  }{\textheight}%
}{0.5ex}}%
\stackon[1pt]{#1}{\scalebox{-1}{\tmpbox}}%
}

\numberwithin{equation}{section}

\allowdisplaybreaks
\newcommand{\Z}{{\mathbb Z}}

\newcommand{\R}{{\mathbb R}}
\newcommand{\N}{{\mathbb N}}
\newcommand{\C}{{\mathbb C}}

\newcommand{\mc}{\mathcal}
\newcommand{\dd}{\mbox{d}}

\newcommand{\eps}{\varepsilon}
\newcommand{\cM}{{\mathcal M}}

\newcommand{\SAP}{\mathcal{S}\hspace*{-2pt}\mathcal{AP}}

\newcommand{\se}{{S^1}}
\newcommand{\s}{{S^p}}
\newcommand{\we}{{W_{\mathcal{A}}^1}}
\newcommand{\w}{{W_{\mathcal{A}}^p}}

\newcommand{\Cu}{C_{\mathsf{u}}}
\newcommand{\Cc}{C_{\mathsf{c}}}

\theoremstyle{plain}
\newtheorem{theorem}{Theorem}[section]
\newtheorem{proposition}[theorem]{Proposition}
\newtheorem{lemma}[theorem]{Lemma}
\newtheorem{corollary}[theorem]{Corollary}

\theoremstyle{definition}
\newtheorem{definition}[theorem]{Definition}
\newtheorem{remark}[theorem]{Remark}

\begin{document}
\title[Stepanov and Weyl almost periodicity in locally compact Abelian groups]{Stepanov and Weyl almost periodicity\\ in locally compact Abelian groups}

\author{Timo Spindeler}
\address{Department of Mathematical and Statistical Sciences, \newline
\hspace*{\parindent}632 CAB, University of Alberta, Edmonton, AB, T6G 2G1, Canada}
\email{spindele@ualberta.ca}

\begin{abstract}
We study Stepanov and Weyl almost periodic functions on locally compact Abelian groups, which naturally appear in the theory of aperiodic order. 
\end{abstract}



\maketitle

\section{Introduction}

Almost periodic functions are natural generalisations of periodic functions, and they appear in many different areas of mathematics. They are well studied objects on the real line because they are used in boundary value problems for differential equations \cite{lz}. A systematic treatment of these functions can be found in \cite{Bes}. Also, there is a well established theory for Bohr almost periodic functions and weakly almost periodic functions on locally compact Abelian groups \cite{Bohr,Eb}. Bohr almost periodic functions turned out to be an important tool in harmonic analysis. Firstly, they are used to study finite-dimensional representations of locally compact Abelian groups \cite[Ch. 16.2]{dix}. Secondly, they describe those autocorrelations which give rise to a pure point diffraction measure; more on that below. However, recent progress in the study of aperiodic order and dynamical systems \cite{LSS,LSS2,SS1} makes it clear that we also need to aim at a better understanding of Stepanov and Weyl almost periodic functions on locally compact Abelian groups. 

Let us be more precise. Meyer sets and cut and project sets are important tools in the study of aperiodic order, since they are used to model quasicrystals. If $\mu$ denotes a translation bounded measure (e.g. the Dirac comb of a Meyer set), then the diffraction measure $\widehat{\gamma_{\mu}}$, which is the Fourier transform of the autocorrelation $\gamma_{\mu}:=\lim_{n\to\infty} |A_n|^{-1} (\mu|_{A_n}*\widetilde{\mu|_{A_n}})$, indicates how (dis)ordered the system that $\mu$ models is \cite{TAO}; note that $(A_n)_{n\in\N}$ is a suitable averaging sequence. Since $\widehat{\gamma_{\mu}}$ is a positive measure, it can be decomposed into three components. The pure point component $(\widehat{\gamma_{\mu}})_{\text{pp}}$ describes the ordered part, whereas the absolutely continuous and singular continuous component $(\widehat{\gamma_{\mu}})_{\text{ac}}$ and $(\widehat{\gamma_{\mu}})_{\text{sc}}$ describe different kinds of disorder. Due to a recent result \cite{LSS}, we are able to characterise measures $\mu$ that give rise to a pure point diffraction measure, i.e. $\widehat{\gamma_{\mu}}=(\widehat{\gamma_{\mu}})_{\text{pp}}$. For obvious reasons, in the next step, we want to find out more about the continuous components. By \cite{Str}, we know that if $\gamma_{\mu}$ is supported inside a Meyer set, its diffraction and each of its components are norm almost periodic measures. This explains the importance of Stepanov almost periodic functions, because an absolutely continuous measure is norm almost periodic if and only if its Radon--Nikodym density is a Stepanov almot periodic function \cite{SS1}. For more background on the connection between norm almost periodicity, cut and project schemes and pure point diffraction, we refer the reader to \cite{BM}. 

Let us briefly mention another application of Stepanov almost periodic functions, namely stable sampling. It is concerned with the problem of reconstrucing a function $f$ from a subset $\Lambda$. Now, a uniformly discrete set $\Lambda$ is a set of stable sampling with respect to $L^p(K)$ and the Stepanov space $S^p(\R^d)$ if every function $f\in L^p(K)$ is the restriction to $K$ of a function $F\in S^p(\R^d)$ whose frequencies belong to $\Lambda$. This was studied in \cite{YM} for the Euclidean space $\R^d$; see \cite{RS} for more general locally compact Abelian groups.

\smallskip

Weyl almost periodic functions also play an important role in aperiodic order. Next to the diffraction measure, certain dynamical systems provide another powerful tool to study aperiodicity. Consider a translation bounded measure $\mu$. It was shown in \cite{LSS} that $\mu$ is Weyl almost periodic if and only if the corresponding hull $\mathbb{X}(\mu)$ has a discrete (or pure point) dynamical spectrum, continuous eigenfunctions and is uniquely ergodic. This result was generalised in \cite{LSS2} from dynamical systems of translation bounded measures $\mathbb{X}(\mu)$, which are of interest in aperiodic order, to arbitrary metric dynamical systems. This is connected to the diffraction measure, because $\mathbb{X}$ has a discrete dynamical spectrum if and only if $\widehat{\gamma_{\mu}}$ is a pure point measure.

By definition, a measure is Stepanov or Weyl almost periodic if the function $f*\mu$ is Stepanov or Weyl almost periodic, for all $f\in \Cc(G)$. Therefore, the results mentioned above motivate to study these kind of almost periodic functions on locally compact Abelian groups.
\bigskip

A function $f:\R\to\R$ is called periodic if there is a number $t\in\R$ such that 
\[
f(x-t)=f(x) \qquad \text{ for all } x\in \R \,, 
\]
which is equivalent to
\[
\|T_tf-f\|_{\infty} = 0,
\]
where $T_tf(x):=f(x-t)$. Next, instead of the supremum norm, one can consider the norm 
\[
\|f\|_{S^1} := \sup_{y\in\R} \frac{1}{\ell} \int_{y}^{y+\ell} |f(x)|\ \dd x \,,
\] 
and define $f\in L_{\text{loc}}^1(\R)$ to be Stepanov almost periodic if the set 
\[
\{t\in\R\ |\ \|f-T_tf\|_{S^1}<\eps \}
\]
is relatively dense in $\R$.

The space of Stepanov almost periodic functions is well studied in the case $G=\R$, see \cite{Bes}. It forms a Banach space with respect to the norm $\|\cdot\|_{S^1}$. One of the main theorems in this area states three equivalent properties which can be used to characterise Stepanov almost periodic functions.  
\begin{theorem}  \label{thm:main_R}
Let $f$ be a locally integrable function on $\R$. Then, the following statements are equivalent:
\begin{enumerate}
\item[(1)] For each $\eps>0$, the set $\{t\in\R\ |\ \|f-T_tf\|_{S^1}<\eps \}$ is relatively dense in $\R$.
\item[(2)] $f$ is the limit of a sequence of trigonometric polynomials with respect to the Stepanov norm $\|\cdot\|_{S^1}$.
\item[(3)] The orbit $\{ T_tf\ |\ t\in\R\}$ is precompact in the $\|\cdot\|_{S^1}$-topology.
\end{enumerate}
\end{theorem}
\noindent Every locally integrable function satisfying one (and thus all) of the above three properties is called a Stepanov almost periodic function. We will prove in this paper that Theorem~\ref{thm:main_R} remains true for second countable locally compact Abelian groups. 

Similarly, one can define that a locally integrable function $f$ is Weyl almost periodic if it is the limit of a sequence of trigonometric polynomials with respect to the seminorm
\[
\|f\|_{W^1} := \lim_{\ell\to\infty} \sup_{y\in \R} \frac{1}{\ell} \int_y^{y+\ell} |f(x)| \ \dd x \,. 
\]
We will see below that Weyl almost periodic functions on locally compact Abelian groups can also be characterised via a relatively dense set of $\eps$-almost periods.

From now on, $G$ denotes a locally comact Abelian group. The corresponding Haar measure is denoted by $\theta_G$.

\section{Stepanov almost periodic functions}

First, we will study Stepanov almost periodic functions on second countable locally compact Abelian groups and establish some of their basic properties. We start by introducing a norm on the set of locally integrable functions $f:G\to\C$.

\begin{definition}
Let $K$ be a precompact set with non-empty interior, and let $1\leqslant p<\infty$. We denote by $BS_{K}(G)$ the space of all $f\in L_{\text{loc}}^p(G)$ such that 
\[
\|f\|_{S_{K}^p(G)} := \sup_{y\in G} \left(\frac{1}{\theta_G(K)} 
\int_{y+K} |f(x)|^p\ \dd \theta_G(x) \right)^{\frac{1}{p}}< \infty\,.
\]
\end{definition}
If $G$ is clear from the context, we write $\|f\|_{S_{K}^p}$ instead of $\|f\|_{S_{K}^p(G)}$, or simply $\|f\|_{S^p}$. The mapping $f\mapsto\|f\|_{S^p}$ defines a norm on $BS_K(G)$. 

Next, let $K_1,\,K_2\subseteq G$ be two compact sets with non-empty interior. We will show that the norms $\|\cdot\|_{S_{K_1}^p}$ and $\|\cdot\|_{S_{K_2}^p}$ are equivalent.

\begin{lemma}  \label{lem:equiv_norms}
For every two compact sets $K_1,\, K_2\subseteq G$ with non-empty interior, there are constants $c_1,c_2>0$ such that 
\[
c_1\, \|f\|_{S_{K_2}^p} \leqslant  \|f\|_{S_{K_1}^p} \leqslant c_2\,  \|f\|_{S_{K_2}^p}
\]
for all $f\in BS_{K_j}^p(G)$, $j\in\{1,2\}$.
\end{lemma}
\begin{proof}
\noindent (i) First note that we can cover $K_1$ with finitely many translates of $K_2$, i.e. there are $n\in\N$ and $z_1,\ldots,z_n\in G$ such that $K_1 \subseteq \bigcup_{j=1}^n (z_j+K_2)$. Then, we have
\begin{align*}
\left(\frac{1}{\theta_G(K_1)} \int_{y+K_1} |f(x)|^p\ \dd \theta_G(x)
 \right)^{\frac{1}{p}}
    &\leqslant \left( \frac{\theta_G(K_2)}{\theta_G(K_1)}\, 
       \frac{1}{\theta_G(K_2)} \int_{\bigcup_{j=1}^n(z_j+K_2)}
       |f(x)|^p\ \dd \theta_G(x) \right)^{\frac{1}{p}}  \\
    &\leqslant \left(n\, \frac{\theta_G(K_2)}{\theta_G(K_1)}\right)^{\frac{1}{p}}  \|f\|_{S_{K_2}^p}
\end{align*}
for all $y\in G$, which proves the second inequality with $c_2:=\left(n\, \frac{\theta_G(K_2)}{\theta_G(K_1)}\right)^{\frac{1}{p}}$.

\medskip

\noindent (ii) The first inequality can be shown similarly.
\end{proof}

Also, it follows from the triangle inequality for $L^p$-spaces that $S_{K}^p(G)$ is a vector space. In fact, the pair $\big(BS_{K}^p(G), \| \cdot\|_{S^p}\big)$ is a Banach space.

\begin{proposition}
The pair $\big(BS_{K}^p(G), \| \cdot\|_{S^p}\big)$ is a Banach space.
\end{proposition}
\begin{proof}
Let $(f_n)_{n\in\N}$ be a Cauchy sequence in $BS_{K}^p(G)$. In that case, we can pick a subsequence $(f_{n_k})_{k\in\N}$ such that 
\[
\|f_{n_{k+1}}-f_{n_k}\|_{S^p} < \frac{1}{2^k}  \quad \ \text{ for all } k\in \N\,.
\]
Let $g(x):=\sum_{k=1}^{\infty} |f_{n_{k+1}}(x)-f_{n_k}(x)|$, for all $x\in G$. Then,
\begin{align*}
\| g\|_{S^p} 
    &\leqslant \sup_{y\in G}\, \liminf_{N\to\infty} \left( \frac{1}{\theta_G(K)}
       \int_{y+K} \Big(\sum_{k=1}^N|f_{n_{k+1}}(x)-f_{n_k}(x)|\Big)^p\ 
       \dd \theta_G(x) \right)^{\frac{1}{p}}   \\
    &= \sup_{y\in G}\, \liminf_{N\to\infty} \frac{1}{\theta_G(K)^{\frac{1}{p}}}
        \Big\| \sum_{k=1}^N |f_{n_{k+1}} -f_{n_k}|\Big\|_{L^p(y+K)}   \\
    &\leqslant \sup_{y\in G}\, \liminf_{N\to\infty} \sum_{k=1}^N
       \frac{1}{\theta_G(K)^{\frac{1}{p}}} \Big\|
        f_{n_{k+1}} -f_{n_k}\Big\|_{L^p(y+K)}   \\
    &\leqslant \liminf_{N\to\infty}\sum_{k=1}^N \Big\|
        f_{n_{k+1}} -f_{n_k}\Big\|_{S^p}  \, \leqslant \, 1\,.
\end{align*}
Therefore, $g\in BS_{K}^p(G)$ and hence $|g|<+\infty$ almost everywhere. Because of this and
\[
f_{n_{N+1}}-f_{n_1} = \sum_{k=1}^{N} \big(f_{n_{k+1}}-f_{n_k} \big) \,,
\]  
there is a set $A$ of full measure such that the sequence $(f_{n_N}(x))_{N\in\N}$ converges to the limit $f_{n_1}(x) + \sum_{k=1}^{\infty} \big(f_{n_{k+1}}(x)-f_{n_k}(x) \big)$ for all $x\in A$. This means that it converges almost everywhere to the function
\[
f(x) := 
\begin{cases}
f_{n_1}(x) + \sum_{k=1}^{\infty} \big(f_{n_{k+1}}(x)-f_{n_k}(x) \big), & \text{ if } x\in A, \\
0, & \text{ otherwise}.
\end{cases}
\]

It remains to show that $(f_{n_N})_{N\in\N}$ converges to $f$ with respect to $\| \cdot \|_{S^p}$ because every Cauchy sequence which contains a converging subsequence converges itself. Let $\eps>0$. Since $(f_n)_{n\in\N}$ is a Cauchy sequence, there is a number $n_0\in \N$ such that 
\[
\|f_n-f_m\|_{S^p} < \eps \quad \ \text{ for all } n,m\geqslant n_0\,.
\]
Finally, this implies (using Fatous Lemma)
\begin{align*}
\|f-f_{n_k}\|_{S^p}
    &= \sup_{y\in G}\left( \frac{1}{\theta_G(K)} \int_{y+K} 
        \lim_{m\to\infty} |f_{n_m}(x)-f_{n_k}(x)|^p
        \ \dd \theta_G(x) \right)^{\frac{1}{p}}  \\
    &\leqslant \sup_{y\in G}\,  \liminf_{m\to\infty} \left( \frac{1}{\theta_G(K)}
        \int_{y+K}|f_{n_m}(x) -f_{n_k}(x)|^p \ \dd \theta_G(x) 
        \right)^{\frac{1}{p}}   \\
    &\leqslant   \liminf_{m\to\infty} \|f_{n_m}-f_{n_k}\|_{S^p} \, <\, \eps  
\end{align*}
for all $k\geqslant n_0$.
\end{proof}

\begin{definition}
Let $K$ and $p$ be as above. A function $f\in L_{\text{loc}}^p(G)$ is called \emph{Stepanov $p$-almost periodic} (with respect to $K$) if, for all $\eps>0$, the set of all $t\in G$ which satisfy
\[
\|f-T_tf\|_{S^p}<\eps
\]
is relatively dense in $G$. The space of all Stepanov $p$-almost periodic functions on $G$ is denoted by $\s(G)$. 
\end{definition}

Again, we may sometimes write $\s$ instead of $\s(G)$. Note that, by Lemma~\ref{lem:equiv_norms}, if $f$ is Stepanov $p$-almost periodic with respect to some compact set $K$ with non-empty interior, it is Stepanov $p$-almost periodic with respect to all compact sets with non-empty interior.

 Let us collect some basic properties of Stepanov almost periodic functions. 

\begin{lemma} \label{lem:2}
Let $f\in \s(G)$.
\begin{enumerate}
\item[(a)] The function $f$ is $S^p$-bounded, i.e. $\|f\|_{S^p}<\infty$.
\item[(b)] The function $f$ is $S^p$-uniformly continous, i.e. for any $\eps>0$, there is a neighbourhood $V$ of $0$ such that
\[
s\in V \implies \|f-T_sf\|_{S^p}<\eps\,.
\]
\item[(c)] If $1\leqslant p'\leqslant p<\infty$, we have $\s(G)\subseteq S^{p'}(G)$. More precisely, we have 
\[
\|f\|_{S^{p'}} \leqslant \|f\|_{S^p} \,. 
\]
\item[(d)] One also has
\[
\s(G)\cap L^{\infty}(G) = \se(G)\cap L^{\infty}(G) \,.
\]
\item[(e)] Let $f^{\dagger}$ and $\widetilde{f}$ be defined via $f^{\dagger}(x):=f(-x)$ and $\widetilde{f}(x):=\overline{f(-x)}$. Let $t\in G$. Then, 
\[ 
f\in\s(G) \implies \overline{f},\, f^{\dagger},\, \widetilde{f},\, |f|,\, T_tf \in\s(G) \,.
\]
\end{enumerate}
\end{lemma}
\begin{proof}
(a) Since $f$ is Stepanov almost periodic, there is a relatively dense set $R$ such that 
\[
\sup_{y\in G} \left( \frac{1}{\theta_G(K)} \int_{y+K} |f(x)-T_tf(x)|^p\ \dd \theta_G(x) \right)^{\frac{1}{p}} < 1
\]
for all $t\in R$. Also, by the relative denseness of $R$, there is a compact set $K'$ such that, for every $y\in G$, there is a $t\in R$ such that $y-t\in K'$. We obtain
\begin{align*}
 \int_{y+K} |f(x)|^p\ \dd \theta_G(x)
    &\leqslant\int_{y+K} |f(x)-T_tf(x)|^p\ \dd \theta_G(x) + \int_{y+K} 
       |T_tf(x)|^p\ \dd \theta_G(x)   \\
    &\leqslant \theta_G(K)+\int_{K'+K} |f(x)|^p\ \dd \theta_G(x)
\end{align*}
uniformly in $y$. Hence, $\|f\|_{S^p}<+\infty$.

\medskip

\noindent (b) Fix $\eps>0$. Then, there is a relatively dense set $R=R(\eps)$ such that 
\[
\sup_{y\in G}\left( \frac{1}{\theta_G(K)}\int_{y+K} |f(x)-T_tf(x)|^p\ \dd \theta_G(x) \right)^{\frac{1}{p}} < \frac{\eps}{3} 
\]
for all $t\in R$. Similar to (a), there is a compact set $K'$ such that, for every $y\in G$, there is a $t\in R$ such that $y-t\in K'$. Then, by \cite[Thm. 1.1.5]{Rud}, there is a a neighbourhood $V$ of $0$ such that
\begin{align*}
\|f-T_sf\|_{S^p}
    &\leqslant \|f-T_tf\|_{S^p} + \|T_tf-T_{s+t}f\|_{S^p} + \|T_{s+t}f-T_sf
       \|_{S^p}  \\
    &< \frac{\eps}{3} + \sup_{y\in G} \left( \frac{1}{\theta_G(K)} \int_{y+K}
       |T_tf(x)-T_{s+t}f(x) |^p\ \dd \theta_G(x) \right)^{\frac{1}{p}} 
        +  \frac{\eps}{3}  \\
    &\leqslant \frac{2\,\eps}{3} + \left(\frac{1}{\theta_G(K)}\int_{K'+K} 
       |f(x)-T_sf(x)|^p \ \dd \theta_G(x)\right)^{\frac{1}{p}} \\
    &< \frac{2\,\eps}{3} + \frac{\eps}{3}\, =\, \eps   \,,   
\end{align*}
whenever $s\in V$. 

\medskip

\noindent (c) This follows from Jensens inequality with the convex function $x\mapsto x^{\frac{p'}{p}}$.

\medskip

\noindent (d) The inclusion $\s(G)\cap L^{\infty}(G) \subseteq \se(G)\cap L^{\infty}(G)$ follows immediately from (c), while $\s(G)\cap L^{\infty}(G) \supseteq \se(G)\cap L^{\infty}(G)$ is a consequence of $\|f\|_{\s} \leqslant \|f\|_{\infty}^{\frac{p-1}{p}} \, \|f\|_{\se}^{\frac{1}{p}}$.

\medskip 

\noindent (e) This follows from the definition.
\end{proof}

\begin{lemma}  \label{lem:bounded_stap}
Let $f\in\s(G)$, and let $\eps>0$. Then, there is a bounded function $b\in\s(G)$ such that
\[
\|f-b\|_{S^p} < \eps\,.
\]
\end{lemma}
\begin{proof}
Let $L\in\N$. Define $f_L:G\to\C$ by
\[
f_L(x) = 
\begin{cases}
f(x),& \text{ if } |f(x)|\leqslant L,  \\
L\, \frac{f(x)}{|f(x)|}, & \text{ otherwise}. 
\end{cases}
\]
Since  
\[
|f_L(x)-f_L(y)|\leqslant |f(x)-f(y)|
\]
for any pair $x,y\in G$, one has $f_L\in\s(G)$. 

Next, let $K$ be a compact subset of $G$. Then, there is a relatively dense set $R$ such that
\[
\|f_L-T_tf_L\|_{S^p} \leqslant \|f-T_tf\|_{S^p} < \frac{\eps}{3}
\]
for all $t\in R$. Since $R$ is relatively dense, there is a compact set $K'$ such that
\begin{align} \label{eq:zurhilfe1}
\|f-f_L\|_{S^p}  \notag
    &\leqslant \|f-T_tf\|_{S^p} + \|T_tf-T_tf_L\|_{S^p} + \|T_tf_L-f_L\|_{S^p} \\
    &\leqslant \frac{\eps}{3} + \left( \frac{1}{\theta_G(K)} \int_{K'+K} 
        |f(x)-f_L(x)|^p\ \dd \theta_G(x) \right)^{\frac{1}{p}}+\frac{\eps}{3} \\
    &= \frac{2\,\eps}{3} +  \left( \frac{1}{\theta_G(K)} \int_{(K'+K)\,\cap\,
        \textsf{S}_L} |f(x)-f_L(x)|^p\ \dd \theta_G(x) \right)^{\frac{1}{p}} \,, \notag         
\end{align}
where $\textsf{S}_L:=\{ x\in G\ |\ |f(x)|>L\}$. It follows from the following standard inequality from measure theory
\[
\mu\left( \{ x\in X\ |\ |f(x)|>L\}\right) \leqslant \frac{1}{L}\, 
\|f\|_{L^1(X,\mu)} 
\] 
that there is $L_0\in\N$ such that
\begin{equation} \label{eq:zurhilfe2}
\left( \frac{1}{\theta_G(K)} \int_{(K'+K)\,\cap\, \textsf{S}_{L_0}} |f(x)-f_{L_0}(x)|^p\ \dd \theta_G(x) \right)^{\frac{1}{p}} < \frac{\eps}{3}  \,.
\end{equation}
Finally, Eqs.~\eqref{eq:zurhilfe1} and \eqref{eq:zurhilfe2} imply
\[
\|f-f_{L_0}\|_{S^p} < \eps \,.
\]
The claim follows for $b:=f_{L_0}$.
\end{proof}

If $G=\R$, Stepanov almost periodic functions are sometimes defined as the closure of the set of trigonometric polynomials with respect to the norm $\|\cdot\|_{S^1}$. Then, it is shown that this is equivalent to the definition given above, which is also equivalent to the statement that the set $\{T_tf\ |\ t\in G\}$ is precompact in the $\|\cdot \|_{S^1}$-topology. The next two propositions will show that this equivalence is still true for second countable locally compact Abelian groups.

\begin{proposition}  \label{prop:main2}
For $f\in L_{\text{loc}}^p(G)$ the following statements are equivalent:
\begin{enumerate}
\item[(1)] The function $f$ is Stepanov $p$-almost periodic.
\item[(2)] The orbit $\{T_tf \ |\  t \in G\}$ is precompact in the $\|\cdot\|_{S^p}$-topology.
\end{enumerate}
\end{proposition}
\begin{proof}
First, let us recall that a set in a complete metric space is precompact if and only if it is totally bounded (that is, for every $\eps>0$, it can be covered by a union of finitely many balls of radius $\eps$). 

\medskip

\noindent (1)$\implies$(2): Let $\eps>0$. Since $f\in\s(G)$, there is a relatively dense set $R$ of $\frac{\eps}{2}$-almost periods. Let $K$ be a corresponding compact set, i.e. $R+K=G$, and let $r>0$ be such that $K\subseteq B_r(0)$. By our choice of the metric, the ball $B_{r}(0)$ is precompact, and hence can be covered by finitely many balls of radius, say, $\delta > 0$. Consequently, due to Lemma~\ref{lem:2}(b), there are elements $z_1,\ldots,z_m\in B_{\delta}(0)$ such that
\[
\inf_{1\leqslant j\leqslant m} \|T_{x_0}f-T_{z_j}f\|_{S^p} < \frac{\eps}{2}
\]
for all $x_0\in B_{r}(0)$. For an arbitrary $x\in G$, there is a $t\in R$ and an  $x_0\in B_r(0)$ such that $x=t+x_0$. In that case, we obtain
\[
\|T_xf-T_{x_0}f\|_{S^p}  = \|T_{t+x_0}f-T_{x_0}f\|_{S^p}  = \|T_{t}f-f\|_{S^p} < \frac{\eps}{2} \,.
\]
Therefore, we have
\[
\inf_{1\leqslant j\leqslant m} \|T_xf-T_{z_j}\|_{S^p} \leqslant  \|T_xf-T_{x_0}f\|_{S^p}+ \inf_{1\leqslant j\leqslant m} \|T_{x_0}f-T_{z_j}f\|_{S^p}  < \frac{\eps}{2} + \frac{\eps}{2} = \eps\,,
\]
and $\{T_tf\ |\ t\in G\}$ is covered by the union of balls of radius $\eps$, centered at $T_{x_j}f$, $1\leqslant j\leqslant m$. Therefore, it is precompact because $L_{\text{loc}}^p(G)$ is a complete metric space. 

\medskip

\noindent (2)$\implies$(1): Fix $\eps>0$. Let $\mathbf{B}_1,\ldots, \mathbf{B}_m$ be balls of radius $\frac{\eps}{2}$ such that $\{T_tf\ |\ t\in G\}$ is covered by the union $\bigcup_{j=1}^m\mathbf{B}_j$. Without loss of generality, we can assume that $\mathbf{B}_j\cap \{T_tf\ |\ t\in G\}\neq\varnothing$. Hence, we can pick $T_{t_j}f\in \mathbf{B}_j$, $1\leqslant j\leqslant m$. The balls of radius $\eps$ centered at $T_{t_j}f$ cover $\{T_tf\ |\ t\in G\}$. 

Now, we claim that every ball $B$ with radius $\rho:=2\, \max_{1\leqslant j\leqslant m} d(t_j,0)$ contains an $\eps$-almost period of $f$. If $B$ is such a ball, denote by $x$ its centre. There is a $j_0$ such that
\[
\|T_xf-T_{t_{j_0}}f\|_{S^p} < \eps\,.
\]
Writing $t=x-t_{j_0}$, it is clear that $t\in B$ and
\[
\|T_tf-f\|_{S^p} = \|T_{x-t_{j_0}}f-f\|_{S^p} = \|T_{x}f-T_{t_{j_0}}f\|_{S^p} < \eps\,.
\]
Therefore, $t$ is an $\eps$-almost period of $f$.
\end{proof}

Sometimes, Stepanov almost periodic functions on $\R$ are defined differently, namely as limits of trigonometric polynomials with respect to $\|\cdot\|_{S^1}$. As we will now see, these two properties are also equivalent for second countable locally compact Abelains groups. 

\begin{proposition} \label{prop:main1}
For $f\in L_{\text{loc}}^p(G)$, the following statements are equivalent:
\begin{enumerate}
\item[(1)] The function $f$ is Stepanov $p$-almost periodic.
\item[(2)] $f$ is the $\|\cdot\|_{S^p}$-limit of a sequence of trigonometric polynomials.
\end{enumerate}
\end{proposition}
\begin{proof}
\noindent (2)$\implies$(1): Let $f$ be the $\|\cdot\|_{S^p}$-limit of a sequence of trigonometric polynomials $(p_n)_{n\in\N}$, and let $\eps>0$. Then, there is a number $n\in\N$ such that 
\[
\|f-p_n\|_{S^p} < \frac{\eps}{3}\,.
\]
Also, since every trigonometric polynomial is Bohr almost periodic and hence Stepanov almost periodic, there is a relatively dense set of $t\in G$ such that
\[
\|p_n-T_tp_n\|_{S^p} < \frac{\eps}{3} \quad \ \text{ for all such } t\in G\,.
\]
This implies
\[
\|f-T_tf\|_{S^p} \leqslant \|f-p_n\|_{S^p} + \|p_n-T_tp_n\|_{S^p} + \|T_tp_n-T_tf_n\|_{S^p} < \frac{\eps}{3} + \frac{\eps}{3} + \frac{\eps}{3} = \eps\,.
\] \smallskip
\noindent (1)$\implies$(2): Let $f\in \s(G)$. Consider the function
\[
f_{\eta}(x) := \frac{1}{\theta_G(B_{\eta}(x))} \int_{B_{\eta}(x)} f(y)\ \dd \theta_G(y)\,.
\] \smallskip
\noindent (i) We will first show that, for fixed $\eta$, the function $f_{\eta}$ is continuous. 

Fix $\eps>0$, and let $q$ be such that $\frac{1}{p} + \frac{1}{q}=1$. By Lemma~\ref{lem:2} and H\"olders inequality, there is $\delta>0$ such that 
\begin{align*}
|f_{\eta}(x)-f_{\eta}(x+h)|
    &= \left| \frac{1}{\theta_G(B_{\eta}(x))}\int_{B_{\eta}(x)} f(y)\ 
        \dd \theta_G(y) - \frac{1}{\theta_G(B_{\eta}(x+h))}
        \int_{B_{\eta}(x+h)} f(y)\ \dd \theta_G(y) \right|  \\
    &\leqslant \frac{1}{\theta_G(B_{\eta}(x))} \int_{B_{\eta}(x)}
       |f(y)-T_{-h}f(y)|\   \dd \theta_G(y)     \\
    &\leqslant \frac{1}{\theta_G(B_{\eta}(x))}\, \theta_G(B_{\eta}
        (x))^{\frac{1}{q}} \, \left( \int_{B_{\eta}(x)}
       |f(y)-T_{-h}f(y)|^p\   \dd \theta_G(y) \right)^{\frac{1}{p}}      \\
    &\leqslant \|f-T_{-h}f\|_{S_{B_{\eta}(x)}^p} \, <\, \eps,      
\end{align*}
whenever $h\in B_{\delta}(0)$. Hence, $f_{\eta}$ is continuous.

\medskip

\noindent (ii) Next, we will show that $f_{\eta}$ is a Bohr almost periodic function.

Let $t$ be a Stepanov $\eps$-almost period of $f$. Then, we have 
\begin{align*}
|f_{\eta}(x)-T_tf_{\eta}(x)|
    &= \left| \frac{1}{\theta_G(B_{\eta}(x))} \int_{B_{\eta}(x)} f(y)\ 
        \dd \theta_G(y) - \frac{1}{\theta_G(B_{\eta}(x-t))}\int_{B_{\eta}(x-t)} 
        f(y)\ \dd \theta_G(y) \right|  \\
    &\leqslant \frac{1}{\theta_G(B_{\eta}(x))}\int_{B_{\eta}(x)} 
       |f(y)-T_{t}f(y)|\   \dd \theta_G(y)     \\
    &\leqslant \|f-T_{t}f\|_{S^p} \, <\, \eps,
\end{align*}
independent of $x$. Consequently, $f_{\eta}$ is Bohr almost periodic.

\medskip

\noindent (iii) We will now prove that $\lim_{\eta\to0} \|f-f_{\eta}\|_{S^p} =0$.

Fix $\eps>0$. By Lemma~\ref{lem:2}, there is $\eta>0$ such that 
\[
h\in B_{\eta}(0) \implies \|f-T_hf\|_{S^p}<\eps\,.
\]
This implies via Minkowski's integral inequality 
\begin{align*}
\|f-f_{\eta}\|_{S^p}
    &= \sup_{y\in G} \left( \frac{1}{\theta_G(K)} \int_{y+K} \bigg| f(x)
         - \frac{1}{\theta_G(B_{\eta}(x))} \int_{B_{\eta}(x)} f(z)\ 
          \dd \theta_G(z) \bigg|^p\ \dd \theta_G(x) \right)^{\frac{1}{p}}  \\
    &=  \sup_{y\in G} \Bigg( \frac{1}{\theta_G(K)}\int_{y+K} \bigg| 
        \frac{1}{\theta_G(B_{\eta}(0))} \int_{B_{\eta}(0)} f(x)\ \dd 
         \theta_G(z) \\
    &\phantom{======} - \frac{1}{\theta_G(B_{\eta}(0))} \int_{B_{\eta}(0)} 
       f(z+x)\ \dd \theta_G(z) \bigg|^p\ \dd \theta_G(x) \Bigg)^{\frac{1}{p}} \\
    &\leqslant  \sup_{y\in G} \frac{1}{\theta_G(K)^{\frac{1}{p}}} \Bigg(  
       \int_{y+K} \Bigg( \frac{1}{\theta_G(B_{\eta}(0))} \int_{B_{\eta}(0)}
       |f(x) - f(x+z)|\ \dd \theta_G(z) \Bigg)^p \ \dd \theta_G(x) \Bigg)^{
        \frac{1}{p}} \\
    &\leqslant \sup_{y\in G} \frac{1}{\theta_G(K)^{\frac{1}{p}}} 
       \frac{1}{\theta_G(B_{\eta}(0))} \int_{B_{\eta}(0)} \Bigg( \int_{y+K} 
       |f(x) - f(x+z)|^p \ \dd \theta_G(x) \Bigg)^{\frac{1}{p}} \ \dd 
       \theta_G(z) \\
    &\leqslant  \frac{1}{\theta_G(B_{\eta}(0))}  \int_{B_{\eta}(0)} 
      \|f-T_{-z}f\|_{S^p} \ \dd \theta_G(z)    \\
    &< \eps\,\frac{1}{\theta_G(B_{\eta}(0))}\,\theta_G(B_{\eta}(0)) \, =\,
       \eps \,.      
\end{align*}

\medskip

\noindent (iv) Finally, we can prove the claim.

Fix $\eps>0$. By step (iii), there is $\eta>0$ such that 
\[
\|f-f_{\eta}\|_{S^p} < \frac{\eps}{2}\,.
\]
Moreover, $f_{\eta}$ is Bohr almost periodic by step (ii). Hence, there is a trigonometric polynomial $P$ such that 
\[
\|f_{\eta}-P\|_{S^p} \leqslant \|f_{\eta}-P\|_{\infty} <  \frac{\eps}{2}\,.
\]
Therefore, we obtain
\[
\|f-P\|_{S^p} \leqslant \|f-f_{\eta}\|_{S^p} + \|f_{\eta}-P\|_{S^p} < \eps\,, 
\]
which finishes the proof.
\end{proof}

Now, Propositions~\ref{prop:main1} and \ref{prop:main2} establish the equivalence of the three statements about Stepanov almost periodic functions mentioned above.

\begin{theorem} \label{thm:mainstap}
For $f\in L_{\text{loc}}^p(G)$ the following statements are equivalent.
\begin{enumerate}
\item[(1)] The function $f$ is Stepanov $p$-almost periodic.
\item[(2)] $f$ is the $\|\cdot\|_{S^p}$-limit of a sequence of trigonometric polynomials.
\item[(3)] The orbit $\{T_tf \ |\  t \in G\}$ is precompact in the $\|\cdot\|_{S^p}$-topology.
\end{enumerate} 
\end{theorem}

\begin{corollary}
The space of Bohr almost periodic functions is dense in $\s(G)$.
\end{corollary}
\begin{proof}
This immediately follows from Proposition~\ref{prop:main1} because every element of $\s(G)$ can be approximated by trigonometric polynomials, which are Bohr almost periodic.
\end{proof}

\begin{remark}
As a consequence of Proposition~\ref{prop:main1}, we find that $(\s(G),\|\cdot\|_{S^p})$ is a normed vector space. In fact, if $f,g\in\s(G)$ and if $(p_n)_{n\in\N},(q_n)_{n\in\N}$ are sequences of trigonometric polynomials which approximate $f$ and $g$, respectively, one has
\[
f + g = \lim_{n\to\infty} p_n + \lim_{n\to\infty} q_n = \lim_{n\to\infty} (p_n + q_n) \in\s(G) \,,
\]
and, in the same way, $c f\in\s(G)$ for every $c\in \C$.
\end{remark}

We can say even more. The pair $(\s(G),\|\cdot\|_{S^p})$ is not only a normed vector space but also a Banach space.

\begin{proposition}
$\big(\s(G), \| \cdot\|_{S^p}\big)$ is a Banach space.
\end{proposition} 
\begin{proof}
Every element $f\in \s(G)$ is an element of $BS^p(G)$ by Lemma~\ref{lem:2}(a). Since the space $(BS^p(G),\|\cdot\|_{S^p})$ is a Banach space, it suffices to show that $\s(G)$ is closed.

In order to do so, let $(f_n)_{n\in\N}$ be a sequence in $\s(G)$ that converges to some function $f$ in the $\|\cdot\|_{S^p}$-topology. First, this means that, given $\eps>0$, there is an element $n_0\in\N$ such that
\[
\|f-f_n\|_{S^p} < \frac{\eps}{3} \quad \ \text{ for all } n\geqslant n_0 \,.
\]
Second, for such an $n$, there is a relatively dense set of elements $t\in G$ such that
\[
\|f_n-T_tf_n\|_{S^p} < \frac{\eps}{3} \quad \ \text{ for all such } t\,.
\]
Finally, one has
\[
\|f-T_tf\|_{S^p} \leqslant \|f-f_n\|_{S^p} + \|f_n-T_tf_n\|_{S^p} + \underbrace{\|T_tf_n-T_tf\|_{S^p}}_{=\,\|f-f_n\|_{S^p}} < \frac{\eps}{3} + \frac{\eps}{3} + \frac{\eps}{3}=\eps\,,  
\] 
for all such $t$, which finishes the proof.
\end{proof}

The space $\big(\s(G), \| \cdot\|_{S^p}\big)$ is closed under addition, but it is not closed under multiplication. To see this, consider the function $\phi(x) = \frac{1}{\sqrt{|x|}}\cdot 1_{[-\frac{1}{2},\frac{1}{2}[}(x)$. Then, the function \[
f(x) = \sum_{k\in\Z} \phi(x+k)
\]
is an element of $\se(\R)$, but $f^2$ is not locally integrable. Hence, $f^2\notin\se(\R)$. However, if one of the two functions is bounded, their product is Stepanov almost periodic.

\begin{proposition} \label{prop:prod_stap}
Let $f,g\in\s(G)$, and let $f$ be bounded. Then, $f\cdot g\in\s(G)$.
\end{proposition}
\begin{proof}
Let $\eps>0$. First, there is a trigonometric polynomial $Q$ such that 
\[
\| g-Q\|_{S^p} < \frac{\eps}{2\, \|f\|_{\infty}} \,.
\] 
Note that we can assume that $\|f\|_{\infty}>0$ (otherwise $f(x)=0$ almost everywhere and the statement is trivially true). As $Q$ is a trigonometric polynomial, it is bounded. Hence, there is a trigonometric polynomial $P$ such that
\[
\|f-P\|_{S^p} < \frac{\eps}{2\, \|Q\|_{\infty}} \,.
\]
Therefore, we obtain
\[
\|fg-PQ\|_{S^p} \leqslant \|fg-fQ\|_{S^p} + \|fQ-PQ\|_{S^p} \leqslant \|f\|_{\infty}\, \|g-Q\|_{S^p} + \|Q\|_{\infty}\, \|f-P\|_{S^p} <\eps \,. 
\]
Since $PQ$ is a trigonometric polynomial, the claim follows.
\end{proof}

If we want the product of two Stepanov almost periodic functions to be an element of $\se(G)$, we can apply H\"olders theorem.

\begin{proposition} \label{prop:holder_s}
Let $f\in\s(G)$, $g\in S^q(G)$ such that $1=\frac{1}{p}+\frac{1}{q}$. Then, $fg\in\se(G)$.
\end{proposition}
\begin{proof}
This is an immediate consequence of H\"olders inequality.
\end{proof}

\section{Weyl almost periodic functions}

After establishing basic properties of Stepanov almost periodic functions in the previous section, we will now focus on Weyl almost periodic functions. The main difference is that we don't consider a fixed compact set $K$ but a specific  sequence of sets - a so called van Hove sequence. 

\begin{definition} 
A sequence $(A_n)_{n\in\N}$ of precompact open subsets of $G$ is called a \textit{van Hove sequence} if, for each compact set $K \subseteq G$, we have
\[
\lim_{n\to\infty} \frac{|\partial^{K} A_{n}|}{|A_{n}|}  =  0 \, ,
\]
where the \textit{$K$-boundary $\partial^{K} A$} of an open set $A$ is defined as
\[
\partial^{K} A := \bigl( \overline{A+K}\setminus A\bigr) \cup
\bigl((\left(G \backslash A\right) - K)\cap \overline{A}\, \bigr) \,.
\]
\end{definition}

Note that every $\sigma$-compact locally compact Abelian group $G$ admits the construction of van Hove sequences \cite{Martin2}.

Let us proceed with the definition of Weyl almost periodic functions.

\begin{definition}
Let $(A_n)_{n\in\N}$ be a van Hove sequence. A function $f\in L_{\text{loc}}^p(G)$ is called \emph{Weyl $p$-almost periodic} (with respect to $(A_n)_{n\in\N}$) if there is a sequence of trigonometric polynomials $(p_m)_{m\in\N}$ such that
\[
\lim_{m\to\infty} \mathcal{M}_{W^p,\mathcal{A}}(f-p_m) := \lim_{m\to\infty} \limsup_{n\to\infty} \, \sup_{y\in G}\left( \frac{1}{\theta_G(A_n)} \int_{y+A_n} |f(x)-p_m(x)|^p\ \dd \theta_G(x) \right)^{\frac{1}{p}} = 0  \,.
\]
We denote the space of Weyl almost periodic functions by $\w(G)$.
\end{definition}

\begin{remark}
A function $f\in L_{\text{loc}}^p(G)$ is Weyl $p$-almost periodic (with respect to $(A_n)_{n\in\N}$) if and only if, for each $\eps>0$, there is a trigonometric polynomial $P$ such that
\[
\limsup_{n\to\infty} \sup_{y\in G} \left(\frac{1}{\theta_G(A_n)} \int_{y+A_n} |f(x)-P(x)|^p\ \dd \theta_G(x)\right)^{\frac{1}{p}} < \eps \,.
\]  
\end{remark}

We will see soon that we can replace `$\limsup_{n\to\infty} \, \sup_{y\in G}$' simply by `$\lim_{n\to\infty}$' in the above definition. To see this, we will first establish that every $f\in\w(G)$ is amenable.  

\begin{proposition} \label{prop:amenable}
Let $(A_n)_{n\in\N}$ be a van Hove sequence. Then, every $f\in\w(G)$ is amenable (with respect to $(A_n)_{n\in\N}$), i.e. the limit
\[
\lim_{n\to\infty}\, \frac{1}{\theta_G(A_{n})} \int_{y+A_{n}} f(x)\ \dd \theta_G(x)
\]
exists uniformly in $y\in G$.
\end{proposition}
\begin{proof}
Fix $\eps>0$. Since $f\in\w(G)$, we can write
\[
f(x) = P(x) + r(x) \,,
\]
where $P$ is a trigonometric polynomial and $r$ satisfies
\[
\limsup_{n\to\infty} \, \sup_{y\in G}\, \left( \frac{1}{\theta_G(A_n)} \int_{y+A_n} |r(x)|^p\ \dd \theta_G(x) \right)^{\frac{1}{p}} < \frac{\eps}{4} \,.
\] 
Every trigonometric polynomial is Bohr almost periodic and (hence) amenable. Therefore, there are numbers $N\in\N$ and $M(P)\in\R$ such that, by Jensens inequality, 
\begin{align*}
\left| \frac{1}{\theta_G(A_n)} \int_{y+A_n} f(x)\ \dd \theta_G(x) - M(P) \right|
    &\leqslant \left| \frac{1}{\theta_G(A_n)} \int_{y+A_n} P(x)\ \dd 
       \theta_G(x) - M(P) \right|  \\
    &\phantom{========} + \frac{1}{\theta_G(A_n)} \int_{y+A_n} |r(x)|\ \dd
       \theta_G(x)    \\
    &\leqslant \left| \frac{1}{\theta_G(A_n)} \int_{y+A_n} P(x)\ \dd 
       \theta_G(x) - M(P) \right|  \\
    &\phantom{========} + \left( \frac{1}{\theta_G(A_n)}  \int_{y+A_n} |r(x)|^p\ 
      \dd \theta_G(x)\right)^{\frac{1}{p}}   \\
    &< \frac{\eps}{4} + \frac{\eps}{4} \, =\,  \frac{\eps}{2}    
\end{align*}
for all $n\geqslant N$, independently of $y$. Consequently, we obtain
\[
\left| \frac{1}{\theta_G(A_{n'})} \int_{y+A_{n'}} f(x)\ \dd \theta_G(x) - \frac{1}{\theta_G(A_{n''})} \int_{y+A_{n''}} f(x)\ \dd \theta_G(x)   \right|
\, < \, \eps
\]
for all $n',n''\geqslant N$, independently of $y$. This finishes the proof.
\end{proof}

\begin{proposition}
Let $(A_n)_{n\in\N}$ be a van Hove sequence, and let $f,g\in \w(G)$. Then, 
\[
f\pm g,\, |f|,\, cf,\, \chi f\in \w(G)
\]
for all $c\in\C$ and $\chi\in \widehat{G}$.
\end{proposition}
\begin{proof}
This is not obvious only for $|f|$: 

Let $\eps>0$. Since $f\in \w(G)$, there is a trigonometric polynomial $P$ such that
\[
\limsup_{n\to\infty} \sup_{y\in G} \left( \frac{1}{\theta_G(A_n)} \int_{y+A_n} 
|f(x)-P(x)|^p\ \dd x  \right)^{\frac{1}{p}} < \frac{\eps}{2} \,.
\]
Note that $|P|$ is Bohr almost periodic, and hence there is a trigonometric polynomial $Q$ such that
\[
\||P|-Q\|_{\infty} < \frac{\eps}{2} \,.
\]
Therefore, we obtain
\begin{align*}
\mathcal{M}_{W^p,\mathcal{A}}(|f|-Q)
    &\leqslant \limsup_{n\to\infty} \sup_{y\in G} \left( \frac{1}{\theta_G(A_n)}
       \int_{y+A_n} ||f(x)|-|P(x)||^p\ \dd x  \right)^{\frac{1}{p}}  \\
    &\phantom{===}+ \limsup_{n\to\infty} \sup_{y\in G} \left( 
       \frac{1}{\theta_G(A_n)} \int_{y+A_n} ||P(x)|-Q(x)|^p\ 
        \dd x  \right)^{\frac{1}{p}}  \\
    &\leqslant \limsup_{n\to\infty} \sup_{y\in G} \left( \frac{1}{\theta_G(A_n)}
       \int_{y+A_n} |f(x)-P(x)|^p\ \dd x  \right)^{\frac{1}{p}} 
       + \||P|-Q\|_{\infty}  \\
    &< \frac{\eps}{2} + \frac{\eps}{2} = \eps \,.  
\end{align*}
\end{proof}

\begin{corollary}  \label{coro:amenable}
Let $f\in\w(G)$, and let $\chi\in\widehat{G}$ be a character. Then, $\overline{\chi}\, f$ is amenable. 
\end{corollary}
\begin{proof}
This is a consequence of the previous two propositions.
\end{proof}

\begin{corollary}  \label{coro:weyl_indep}
Let $f\in L_{\text{loc}}^p(G)$. Then, the function $f$ is Weyl $p$-almost periodic if and only if there is a sequence of trigonometric polynomials $(p_m)_{m\in\N}$ such that
\begin{equation} \label{eq:limlim}
\lim_{m\to\infty} \lim_{n\to\infty} \left( \sup_{y\in G}\frac{1}{\theta_G(A_n)} \int_{y+A_n} |f(x)-p_m(x)|^p\ \dd \theta_G(x)\right)^{\frac{1}{p}} = 0  \,.
\end{equation}
\end{corollary}
\begin{proof}
It is easy to see that $|f-P|$ is Weyl $p$-almost periodic if $P$ is a trigonometric polynomial and $f\in\w(G)$. Hence, $|f-P|$ is amenable by Proposition~\ref{prop:amenable}, and the claim follows.

The other direction is obvious.
\end{proof}

\begin{remark}
Let us define 
\[
\|f\|_{\w} := \lim_{n\to\infty}  \left( \frac{1}{\theta_G(A_n)} \int_{y+A_n} |f(x)|^p \ \dd \theta_G(x) \right)^{\frac{1}{p}} \,,
\]
for an arbitrary $y\in G$. Then, Eq.~\eqref{eq:limlim} becomes
\[
\lim_{m\to\infty} \|f-p_m\|_{\w} =0 
\] 
because $\|f \|_{\w}=\mathcal{M}_{W^p}(f)$ for all $f\in \w(G)$, see Proposition~\ref{prop:amenable} and Corollary~\ref{coro:weyl_indep}.
Note that, in contrast to $(\s(G),\|\cdot\|_{S^p})$, the pair $(\w(G),\|\cdot\|_{\w})$ is not a complete space. Moreover, $\|\cdot\|_{W^p}$ is not a norm but only a seminorm (in general) because $\|f\|_{W^p}=0$ for all $f\in L^p(\R^d)$. This leads to another difference between Stepanov and Weyl almost periodic functions. While $\|f-g\|_{\s}=0$ implies that $f$ and $g$ coincide almost everywhere, $f$ and $g$ can differ on a set of positive (or even infinite) measure when $\|f-g\|_{\w}=0$.
\end{remark}

\begin{proposition} \label{prop:w<s}
We have $\s(G)\subseteq \w(G)$. More precisely, there is a compact set $C$ such that
\[
\|f\|_{\w} \leqslant \|f\|_{S_C^p} 
\]
for every $f\in\s(G)$.
\end{proposition}
\begin{proof}
Let $f\in\s(G)$, and let $(A_n)_{n\in\N}$ be a van Hove sequence. Consider the measure $\mu:=|f|^p\, \theta_G$. By \cite[Lem. 3.3]{SS3}, there is a compact set $C$ such that
\[
\limsup_{n\to\infty} \frac{1}{\theta_G(A_n)} \int_{y+A_n} |f(x)|^p\ \dd \theta_G(x) = \limsup_{n\to\infty} \frac{|\mu|(y+A_n)}{\theta_G(y+A_n)}
\leqslant \frac{\|\mu\|_C}{\theta_G(C)} = \|f\|_{S_C^p}^p 
\] 
for all $y\in G$.
Since every $f\in\w(G)$ is amenable by Proposition~\ref{prop:amenable}, we can replace `$\limsup$' by `$\lim$' in the above equation, and we obtain
\[
\|f\|_{\w} \leqslant \|f\|_{S_C^p} \,,
\]
which finishes the proof.
\end{proof}

\begin{corollary}
Every $f\in \s(G)$ is amenable. 
\end{corollary}
\begin{proof}
This is an immediate consequence of Propositions~\ref{prop:amenable} and \ref{prop:w<s}.
\end{proof}

The next goal is to establish the analogon of Proposition~\ref{prop:main1} for Weyl $p$-almost periodic functions. Thus, we will show that Weyl almost periodicity can also be characterised by relatively dense sets of almost periods. For this, we will need the following lemmas. 

\begin{lemma} \label{lem:w^p_uc}
Let $f\in L_{\text{loc}}^p(G)$ be such that, for all $\eps'>0$, there is a relatively dense set $R$ and a number $N\in\N$ such that
\[
\|f-T_tf\|_{S_{A_n}^p} < \eps'
\]
for all $t\in R$ and all $n\geqslant N$. Then, for all $\eps>0$, there is a neighbourhood $V$ of $0$ such that 
\[
\|f-T_sf\|_{S_{A_N}^p} < \eps
\]
for all $s\in V$.
\end{lemma}
\begin{proof}
Fix $\eps>0$. By assumption, there is a relatively dense set $R$ and a number $N\in\N$ such that $\|f-T_tf\|_{S_{A_n}^p} < \frac{\eps}{3}$ for all $n\geqslant N$ and all $t\in R$. So, we can write $G=R+K$ for some compact set $K$. Due to \cite[Thm. 1.1.5]{Rud}, there is a neighbourhood $V$ of $0$ such that
\[
\left( \frac{1}{\theta_G(A_N)} 
       \int_{K+A_N} |f(x)-T_{s}f(x)|^p\ \dd \theta_G(x) 
       \right)^{\frac{1}{p}} < \frac{\eps}{3}
\]
for all $s\in V$. Consequently, one has
\begin{align*}
\|f-T_sf\|_{S_{A_N}^p}
    &\leqslant \|f-T_tf\|_{S_{A_N}^p} + \|T_tf-T_{t+s}f\|_{S_{A_N}^p} + 
      \|T_{t+s}f-T_sf\|_{S_{A_N}^p}   \\
    &< \frac{\eps}{3} + \sup_{y\in G} \left( \frac{1}{\theta_G(A_N)} 
       \int_{y+A_N} |T_tf(x)-T_{s+t}f(x)|^p\ \dd \theta_G(x) 
       \right)^{\frac{1}{p}} + \frac{\eps}{3}  \\
    &< \frac{\eps}{3} + \left( \frac{1}{\theta_G(A_N)} 
       \int_{K+A_N} |f(x)-T_{s}f(x)|^p\ \dd \theta_G(x) 
       \right)^{\frac{1}{p}} + \frac{\eps}{3}   \\
    &< \eps      
\end{align*}
for all $s\in V$, which completes the proof.
\end{proof}

\begin{lemma} \label{lem:K}
Let $f\in L_{\text{loc}}^p(G)$ be such that, for every $\eps>0$, there is a relatively dense set $R$ and a number $N\in\N$ such that 
\[
 \|f-T_tf\|_{S_{A_n}^p} < \eps 
\]
for all $n\geqslant N$ and all $t\in R$. Fix $\eps'>0$ and $N'\in\N$. Then, there is a function $\mathcal{K}:G\to\C$ with the following properties:
\begin{enumerate}
\item[$\bullet$] $\mathcal{K}\geqslant 0$,
\item[$\bullet$] $\mathcal{K}$ is bounded,
\item[$\bullet$] $\liminf_{n\to\infty} \frac{1}{\theta_G(A_n)} \int_{A_n} \mathcal{K}(z)\ \dd \theta_G(z)=1$,
\item[$\bullet$] $\mathcal{K}(z)\neq 0$ only if $\|f-T_{-z}f\|_{S_{A_{N'}}^p} < \eps'$.
\end{enumerate}
\end{lemma}
\begin{proof}
Fix $\eps'>0$. By assumption, there is a number $N\in\N$ such that
\[
A:=\{ z\in G\ |\ \|f-T_{-z}f\|_{S_{A_N}^p} <\eps' \} 
\]
and 
\[
B:=\{ z\in G\ |\ \|f-T_{-z}f\|_{S_{A_N}^p} <\frac{\eps'}{2} \}
\]
are relatively dense. Thus, there is a compact set $K$ with $G=K+B$. Moreover, by Lemma~\ref{lem:2}, there is an open neighbourhood $V$ of $0$ such that
\[
\|f-T_{-t}f\|_{S_{A_N}^p}<\frac{\eps'}{2}  \qquad \text{ for all } t\in V\,.
\]
Note that $B+V\subseteq A$. Since $K$ is compact, it can be covered by finitely many translates of $V$, i.e. $K\subseteq \bigcup_{j=1}^{\ell} (V+t_j)$, $t_1,\ldots,t_{\ell}\in G$. Hence, one has
\[
G=B+K\subseteq B + \bigcup_{j=1}^{\ell} (V+t_j) = \bigcup_{j=1}^{\ell} ((V+B)+t_j) \subseteq \bigcup_{j=1}^{\ell} (A+t_j) \,.
\]
This implies that
\[
c:= \liminf_{n\to\infty} \frac{\theta_G(A\cap A_n)}{\theta_G(A_n)}
  =\liminf_{n\to\infty}\frac{1}{\theta_G(A_n)} \int_{A_n} 1_A(x)\ \dd\theta_G(x) >0 \,.
\]
Now, the function $\mathcal{K}(x)=\frac{1}{c} \, 1_A(x)$ satisfies the properties.
\end{proof}

\begin{proposition}  \label{prop:weap_eps_char}
Let $f\in L_{\text{loc}}^p(G)$. Then, the following statements are equivalent:
\begin{enumerate}
\item[(1)] $f$ is Weyl $p$-almost periodic.
\item[(2)] For every $\eps>0$, there is a relatively dense set $R$ and a number $N\in\N$ such that 
\[
 \|f-T_tf\|_{S_{A_n}^p} < \eps 
\]
for all $n\geqslant N$ and all $t\in R$.
\end{enumerate}
\end{proposition}
\begin{proof}
(1)$\implies$(2): Fix $\eps>0$. Since $f\in\w(G)$, there is a trigonometric polynomial $P$ such that $\|f-P\|_{\w} < \frac{\eps}{3}$. In particular, there is a number $N\in\N$ such that
\[
\|f-P\|_{S_{A_n}^p} = \sup_{y\in G} \left( \frac{1}{\theta_G(A_n)} \int_{y+A_n} |f(x)-P(x)|^p\ \dd \theta_G(x) \right)^{\frac{1}{p}} <  \frac{\eps}{3} 
\] 
for all $n\geqslant N$, independently of $y$. This and the fact that every trigonometric polynomial is Bohr almost periodic imply that there is a relatively dense set $R$ such that
\begin{align*}
\|f-T_tf\|_{S_{A_n}^p}
    &\leqslant \|f-P\|_{S_{A_n}^p} + \|P-T_tP\|_{S_{A_n}^p} 
       + \|T_tP-T_tf\|_{S_{A_n}^p}  \\
    &< \frac{\eps}{3} + \|P-T_tP\|_{\infty} +  \frac{\eps}{3} \, < \, \eps
\end{align*}
for all $n\geqslant N$ and all $t\in R$.

\medskip

\noindent (2)$\implies$(1): Fix $\eps>0$ and $y\in G$. By assumption, there is a number $N\in\N$ and a relatively dense set $R$ such that 
\[
\|f-T_tf\|_{S_{A_n}^p} = \sup_{y\in G} \left( \frac{1}{\theta_G(A_n)} \int_{y+A_n} |f(x)-T_tf(x)|^p\ \dd \theta_G(x) \right)^{\frac{1}{p}} <  \frac{\eps}{3} 
\] 
for all $n\geqslant N$ and $t\in R$, independently of $y$.

Next, we consider a function $\mathcal{K}:G\to\C$ with the following properties:
\begin{enumerate}
\item[$\bullet$] $\mathcal{K}\geqslant 0$,
\item[$\bullet$] $\mathcal{K}$ is bounded,
\item[$\bullet$] $\liminf_{n\to\infty} \frac{1}{\theta_G(A_n)} \int_{A_n} \mathcal{K}(z)\ \dd \theta_G(z)=1$,
\item[$\bullet$] $\mathcal{K}(z)\neq 0$ only if $\|f-T_{-z}f\|_{S_{A_N}^p} < \eps$.
\end{enumerate}
Such a function exists by Lemma~\ref{lem:K}. Moreover, we define the function $\phi:G\to\C$ by
\[
\phi(x) := \liminf_{n\to\infty} \frac{1}{\theta_G(A_n)} \int_{A_n} f(x+z)\, \mathcal{K}(z)\ \dd \theta_G(z) \,.
\]

\medskip 

\noindent (i) First, we will show that $\phi$ is continuous.

To see this, only note that, by Lemma~\ref{lem:w^p_uc}, for every $\eps'>0$, there is a neighbourhood $V$ of $0$ such that
\begin{equation} \label{eq:help1}
\begin{split}
|\phi(x)-T_{-\delta}\phi(x)|
    &\leqslant \liminf_{n\to\infty} \frac{1}{\theta_G(A_n)} \int_{A_n} 
      |f(x+z) -f(x+\delta+z)|\, \mathcal{K}(z)\ \dd\theta_G(z) \\
    &\leqslant \|\mathcal{K}\|_{\infty} \, \|f-T_{-\delta}f\|_{W^p} \\
    &\leqslant c\, \|\mathcal{K}\|_{\infty}\, \|f-T_{-\delta}f\|_{S_{A_N}^p} \\
    &< \eps' \,,  
\end{split}
\end{equation} 
whenever $\delta\in V$, where we made use of Proposition~\ref{prop:w<s} and Lemma~\ref{lem:equiv_norms} in the penultimate step. Hence, $\phi$ is continuous.

\medskip

\noindent (ii) Additionally, $\phi$ is Bohr almost periodic.

This immediately follows from Eq.~\eqref{eq:help1}, since $\delta=-t$ gives
\[
|\phi(x)-T_{t}\phi(x)| \leqslant \|\mathcal{K}\|_{\infty}\, \|f-T_{t}f\|_{W^p}
\]
for all $x\in G$.

\medskip 

\noindent (iii) We have $\|f-\phi\|_{\w} < \eps$.

This is a consequence of
\begin{align*}
{}
    &\frac{1}{\theta_G(A_N)} \int_{y+A_N} |f(x)-\phi(x)| \ \dd \theta_G(x)  \\
    &\phantom{====}\leqslant \frac{1}{\theta_G(A_N)} \int_{y+A_N} 
        \liminf_{n\to\infty} 
       \frac{1}{\theta_G(A_n)} \int_{A_n} |f(x)-f(x+z)|\, \mathcal{K}(z)\ \dd 
       \theta_G(z) \ \dd \theta_G(x)   \\
    &\phantom{====}\leqslant \liminf_{n\to\infty} \frac{1}{\theta_G(A_n)}
       \int_{A_n} \frac{1}{\theta_G(A_N)} \int_{y+A_N}  |f(x)-f(x+z)|\,
        \dd \theta_G(x)\ \mathcal{K}(z) \ \dd \theta_G(z)  \\
    &\phantom{====}\leqslant \liminf_{n\to\infty} \frac{1}{\theta_G(A_n)}
       \int_{A_n}  \|f-T_{-z}f\|_{S_{A_N}^p}\ 
       \mathcal{K}(z) \ \dd \theta_G(z) \\
    &\phantom{====}< \eps \, \liminf_{n\to\infty} \frac{1}{\theta_G(A_n)}
       \int_{A_n}  \mathcal{K}(z) \ \dd \theta_G(z) \, =\, \eps \,,
\end{align*}
where we applied Fatous lemma and Fubinis theorem.

\medskip 

\noindent (iv) Finally, we can prove the claim.

This follows exactly as in the last step of the proof of Proposition~\ref{prop:main1}.
\end{proof}

\begin{remark} \label{rem:w_ap}
Let $f\in L_{\text{loc}}^p(G)$. In this case, the property

\smallskip 

\noindent (I) \textit{For every $\eps>0$, there is a relatively dense set $R$ and a number $N\in\N$ such that 
\[
 \|f-T_tf\|_{S_{A_n}^p} < \eps 
\]
for all $n\geqslant N$ and all $t\in R$.}

\smallskip

\noindent is stronger than 

\smallskip

\noindent (II) \textit{For every $\eps>0$, there is a relatively dense set $R$ such that 
\[
 \lim_{n\to\infty}\|f-T_tf\|_{S_{A_n}^p} < \eps 
\]
for all $t\in R$.}

\smallskip

\noindent 
The reason for this is that (II) is equivalent to

\smallskip

\noindent (II') For every $\eps>0$, there is a relatively dense set $R$ such that, for all $t\in R$, there is a number $N\in\N$ such that
\[
 \|f-T_tf\|_{S_{A_n}^p} < \eps 
\]
for all $n\geqslant N$. 

\smallskip

The difference is that $N$ can be chosen independently of $t$ in (I), but it will depend on $t$ in (II'). For this reason, some people prefer to call elements $f\in\w(\R)$ \textit{equi-Weyl almost periodic}.  
\end{remark}

Next, we can state the analoga of Lemma~\ref{lem:bounded_stap}, Proposition~\ref{prop:prod_stap} and Proposition~\ref{prop:holder_s}.

\begin{lemma}  \label{lem:bounded_weap}
Let $f\in\w(G)$, and let $\eps>0$. Then, there is a bounded function $b\in\w(G)$ such that
\[
\|f-b\|_{\w} < \eps\,.
\]
\end{lemma}
\begin{proof}
This is proved like Lemma~\ref{lem:bounded_stap} using the characterisation from Proposition~\ref{prop:weap_eps_char}.
\end{proof}

\begin{proposition} \label{prop:holder_w}
Let $f\in\w(G)$, $g\in W_{\mathcal{A}}^q(G)$ such that $1=\frac{1}{p}+\frac{1}{q}$. Then, $fg\in\we(G)$.
\end{proposition}
\begin{proof}
This is an immediate consequence of H\"olders inequality.
\end{proof}

At the end of this section, let us briefly mention some additional properties of bounded Weyl almost periodic functions.

\begin{proposition} \label{prop:prod_weap}
Let $f,g\in\w(G)$, and let $f$ be bounded. Then, $f g\in\w(G)$.
\end{proposition}
\begin{proof}
Simply imitate the proof of Proposition~\ref{prop:prod_stap}.
\end{proof}

\begin{proposition} (\cite[Prop. 4.11]{LSS})
Let $f : G\to \C$ be a bounded and measurable function.
Let  $\mathcal{A}$ and $\mathcal{B}$ be van Hove sequences. Then, $f$ belongs to
$\w(G)$ if and only if it belongs to $Wap_{\mathcal{B}}^p (G)$. If
$f$ belongs to $\w (G)$ and $Wap_{\mathcal{B}}^p (G)$, then $ \| f
\|_{\w} = \| f \|_{W_{\mathcal{B}^p}}$ holds.
\end{proposition}

The previous proposition says that a bounded Weyl $p$-almost periodic function is independent of the choice of the van Hove sequence. 

\section{Convolutions with Stepanov and Weyl almost periodic functions} \label{sec:conv}

In this section, we will have a very brief look at convolutions with Stepanov almost periodic functions. We will consider the following three kinds of convolutions:
\begin{enumerate}
\item[(1)] If $f$ and $g$ are functions from $G$ to $\C$, we define
\[
(f*g)(x) := \int_G f(x-y)\, g(y)\ \dd \theta_G(y)\,,
\]
whenever the integral exists.
\item[(2)] If $\mu$ is a measure on $G$, we define
\[
(f*\mu)(x) := \int_G f(x-y)\ \dd\mu(y)\,,
\]
whenever the integral exists.
\item[(3)] Last, we define the \textit{Eberlein convolution} of $f$ and $g$ (with respect to the van Hove sequence $(A_n)_{n\in\N}$) by
\[
(f\circledast g)(x) := \lim_{n\to\infty} \frac{1}{\theta_G(A_n)} \int_{A_n} f(x-y)\, g(y)\ \dd\theta_G(y)\,,
\]
whenever the limit exists.
\end{enumerate}

\begin{proposition}
Let $f\in\s(G)$, and let $\mu$ be a finite measure. Then, $f*\mu\in\s(G)$.
\end{proposition}
\begin{proof}
The statement follows from Minkowskis inequality for integrals because
\begin{align*}
\|f*\mu-T_t(f*\mu)\|_{S^p} 
    &= \|(f-T_tf)*\mu\|_{S^p}  \\
    &= \sup_{y\in G} \left( \frac{1}{\theta_G(K)} \int_{y+K} |((f-T_tf)*\mu)(x)
       |^p\ \dd \theta_G(x) \right)^{\frac{1}{p}}  \\
    &=  \sup_{y\in G} \left( \frac{1}{\theta_G(K)}\int_{y+K}\left|\int_G(f-T_tf)
       (x-z)\ \dd \mu(z)\right|^p\ \dd \theta_G(x) \right)^{\frac{1}{p}}  \\ 
    &\leqslant  \sup_{y\in G} \int_G \left( \frac{1}{\theta_G(K)}\int_{y+K} 
       |(f-T_tf)(x-z)|^p \ \dd \theta_G(x) \right)^{\frac{1}{p}}\ \dd\mu(z) \\
    &\leqslant \int_G \|f-T_tf\|_{S^p} \ \dd\mu(z) \\     
    &= \mu(G)\, \|f-T_tf\|_{S^p}\,.   \qedhere
\end{align*}   
\end{proof}

The next corollary is an immediate consequence.

\begin{corollary} \label{coro:conv_l1}
Let $f\in\s(G)$, and let $g\in L^1(G)$. Then, $f*g\in\s(G)$. 
\end{corollary}

In general, the convolution of two Weyl almost periodic functions does not exist. For this reason, we consider its averaged version - the Eberlein convolution.

\begin{lemma}
Let $f\in\w(G)$, and let $g\in W^q(G)$ for $q\geqslant 1$ with $\frac{1}{p}+\frac{1}{q}=1$. Then, the limit $f\circledast g$ exists.
\end{lemma}
\begin{proof}
Fix $\eps>0$. Since $f\in\w(G)$ and $g\in W^q(G)$, we can write
\[
f(x) = P(x) + r(x) \qquad \text{ and } \qquad g(x) = Q(x) + s(x) \,,
\]
where $P,Q$ are trigonometric polynomials and $r$ and $s$ satisfy
\[
\limsup_{n\to\infty} \, \sup_{y\in G}\, \left( \frac{1}{\theta_G(A_n)} \int_{y+A_n} |r(x)|^p\ \dd \theta_G(x) \right)^{\frac{1}{p}} < \frac{\eps}{8\, \|Q\|_{\infty}} 
\] 
and 
\[
\limsup_{n\to\infty} \, \sup_{y\in G}\, \left( \frac{1}{\theta_G(A_n)} \int_{y+A_n} |s(x)|^q\ \dd \theta_G(x) \right)^{\frac{1}{q}} < \frac{\eps}{8\, \|P\|_{\infty}} \,.
\] 
Therefore, there is an $N\in\N$ such that (via the Jensen and H\"older inequality)
\begin{align*}
{}
    & \left| \frac{1}{\theta_G(A_n)} \int_{A_n} f(x-z)\, g(z)\ \dd\theta_G(z) - (P\circledast Q)(x) \right| \\
    &\phantom{++++}\leqslant \left| \frac{1}{\theta_G(A_n)} \int_{A_n} P(x-z)\, Q(z)\ \dd\theta_G(z)  - 
          (P\circledast Q)(x) \right| \\
    &\phantom{========} + \left| \frac{1}{\theta_G(A_n)} \int_{A_n} \Big( Q(y)\, r(x-z) +  P(x-z)\, s(z)
           + r(x-z)\, s(z) \Big) \dd\theta_G(z) \right|  \\
    &\phantom{++++}< \frac{\eps}{8} + \|Q\|_{\infty} \, \frac{1}{\theta_G(A_n)} \int_{A_n} |r(x-z)|\ 
    \dd\theta_G(z)  + \|P\|_{\infty}\, \frac{1}{\theta_G(A_n)} \int_{A_n} |s(z)|\ 
        \dd\theta_G(z)   \\
    &\phantom{========} +  \frac{1}{\theta_G(A_n)} \int_{A_n} |r(x-z)\, s(z)|\ \dd\theta_G(z)  \\
    &\phantom{++++}\leqslant \frac{\eps}{8}  +  \|Q\|_{\infty} \, \left(\frac{1}{\theta_G(A_n)} \int_{A_n} 
        |r(x-z)|^p\ \dd\theta_G(z)  \right)^{\frac{1}{p}}  \\
    &\phantom{========} + \|P\|_{\infty}\, \left( \frac{1}{\theta_G(A_n)} \int_{A_n} |s(z)|^q\ 
        \dd\theta_G(z) \right)^{\frac{1}{q}}  +  \frac{1}{\theta_G(A_n)} \int_{A_n} |r(x-z)\, s(z)|\ 
        \dd\theta_G(z)       \\
    &\phantom{++++}< \frac{3\,\eps}{8} + \left( \frac{1}{\theta_G(A_n)}
        \int_{A_n} |s(z)|^q\ \dd\theta_G(z) \right)^{\frac{1}{q}} \left(\frac{1}{\theta_G(A_n)} \int_{A_n} 
        |r(x-z)|^p\ \dd\theta_G(z)  \right)^{\frac{1}{p}}  \\
    &\phantom{++++} < \frac{3\,\eps}{8}  + \frac{\eps^2}{16\, \|Q\|_{\infty}\, \|P\|_{\infty}}\,  < \,     
         \frac{\eps}{2}     
\end{align*}
for all $n\geqslant N$. Consequently, we obtain
\[
\left| \frac{1}{\theta_G(A_{n'})} \int_{y+A_{n'}} f(x-y)\, g(y)\ \dd \theta_G(y) - \frac{1}{\theta_G(A_{n''})} \int_{y+A_{n''}} f(x-y)\, g(y)\ \dd \theta_G(y)   \right|
\, < \, \eps
\]
for all $n',n''\geqslant N$. This finishes the proof.
\end{proof}

\begin{proposition} \label{prop:conv_s_ap}
Let $f\in\w(G)$, and let $g\in W^q(G)$ for $q\geqslant 1$ with $\frac{1}{p}+\frac{1}{q}=1$. Then, $f\circledast g$ is a Bohr almost periodic function.
\end{proposition}
\begin{proof}
By definition, there are sequences of trigonometric polynomials $(P_n)_{n\in\N}$ and $(Q_n)_{n\in\N}$ such that
\[
\lim_{n\to\infty} \|f-P_n\|_{W^p} = 0 \qquad \text{ and } \qquad \lim_{n\to\infty} \|g-Q_n\|_{W^q} = 0 \,.
\]
This immediately implies that there is a constant $c>0$ such that $\|Q_n\|_{W^q}\leqslant c$ for all $n\in\N$. Applying the H\"older inequality, we get
\begin{align*}
\|f\circledast g-P_n\circledast Q_n\|_{\infty} 
    &\leqslant \|f\circledast (g- Q_n)\|_{\infty} + \|(f -P_n)\circledast Q_n\|_{\infty} \\
    &\leqslant \|f\|_{W^p}\, \|g-Q_n\|_{W^q} + \|f-P_n\|_{W^p}\, \|Q_n\|_{W^q} \,.
\end{align*}
Thus, the sequence $(P_n\circledast Q_n)_{n\in\N}$ converges uniformly to $f\circledast g$. The claim follows because $P_n\circledast Q_n$ is Bohr almost periodic for all $n\in\N$, see \cite[Thm. 4.6.3]{TAO2}, and uniform limits of Bohr almost periodic functions are Bohr almost periodic \cite[Prop. 4.3.4]{TAO2}.
%Proposition~\ref{prop:conv_stap} implies that $f\circledast g\in\s(G)$. By Proposition~\ref{prop:s_apcu}, it remains to show that $f\circledast g$ is uniformly continuous. 
%
%Fix $\eps>0$. Applying \cite[Prop. 4.5.9(vii)]{TAO2}, Proposition~\ref{prop:w<s} and Lemma~\ref{lem:2}, there is a compact set $C$ and a neighbourhood $V$ of $0$ such that
%\begin{align*}
%|(f\circledast g)(x) - (f\circledast g)(y)|
%    &\leqslant \lim_{n\to\infty} \frac{1}{\theta_G(A_n)} \int_{s+A_n}
%      |f(x-z) - f(y-z)|\, |g(z)|\ \dd \theta_G(z)  \\
%    &= M\Big[|f(x-\cdot)-f(y-\cdot)|\,|g|\Big]  \\
%    &\leqslant  \bigg(M\Big[|f(x-\cdot)-f(y-\cdot)|^p\Big]\bigg)^{\frac{1}{p}} 
%       \, \bigg(M\Big[|g|^q\Big]\bigg)^{\frac{1}{q}} \\
%    &= \|T_{x-y}f^{\dagger}-f^{\dagger}\|_{W^p} \, \bigg(M\Big[|g|^q\Big]
%       \bigg)^{\frac{1}{q}}  \\
%    &\leqslant  \|T_{x-y}f^{\dagger}-f^{\dagger}\|_{S_C^p} \, 
%       \|g\|_{W^q} \, < \, \eps \,,
%\end{align*}
%whenever $x-y\in V$. The claim follows.
\end{proof}

\section{Fourier--Bohr series}

\begin{definition}
Let $f\in\w(G)$, let $y\in G$, and let $(A_n)_{n\in\N}$ be a van Hove sequence. We define the \emph{Fourier--Bohr coefficients} of $f$ by 
\[
c_{\chi}(f):= M(\overline{\chi}\, f) := \lim_{n\to\infty} \frac{1}{|A_n|} \int_{y+A_n} \overline{\chi(x)}\, f(x)\ \dd x 
\]
if the limit exists.
Moreover, we call the formal sum
\[
\sum_{\chi\in\widehat{G}} c_{\chi}(f)\, \chi
\]
the \emph{Fourier--Bohr series} of $f$. 
\end{definition}

\begin{proposition}
Let $f\in\w(G)$. 
\begin{enumerate}
\item[(a)] The Fourier--Bohr coefficients $c_{\chi}(f)$ exist uniformly in $y$ for all $\chi\in\widehat{G}$. If in addition $f$ is uniformly continuous and bounded, $c_{\chi}(f)$ is also independent of $(A_n)_{n\in\N}$. 
\item[(b)] The set $\{ \chi\in\widehat{G}\ |\ c_{\chi}(f) \neq 0\}$ is at most countable.
\end{enumerate} 
\end{proposition}
\begin{proof}
(a) This immediately follows from Corollary~\ref{coro:amenable}. 

\medskip

\noindent (b) By definition, $f$ can be approximated by a sequence $(p_n)_{n\in\N}$ of trigonometric polynomials. For fixed $n\in\N$, the value $c_{\chi}(p_n)$ differs from $0$ only for a finite number of $\chi$, since $M(\chi)=0$ for all non-trivial characters $\chi$. From this, we can infer the claim.
\end{proof}

\begin{corollary}
The Fourier--Bohr series exists for all $f\in\w(G)$.
\end{corollary}

Now, since $c_{\chi}(f)$ exists for all $\chi\in\widehat{G}$, the next result follows from standard techniques.

\begin{proposition} \label{prop:bessel}
Let $f\in W^2(G)$, and let $\chi_1,\ldots,\chi_n\in\widehat{G}$ be distinct characters. Then, one has
\[
M\bigg(\Big|f-\sum_{j=1}^n c_{\chi_j}(f)\, \chi_j\Big|^2\bigg) = M\big[|f|^2\big] 
-  \sum_{j=1}^n |c_{\chi_j}(f)|^2 \,.
\]
Consequently, $\sum_{j=1}^n |c_{\chi_j}(f)|^2 \leqslant  M\big[|f|^2\big]$. Hence, (we obtain one more time that) the set $\{\chi\in\widehat{G}\ |\ c_{\chi}(f) \neq 0\}$ is at most countable, and one has Bessels inequality
\[
\sum_{j=1}^{\infty} |c_{\chi_j}(f)|^2 \leqslant  M\big[|f|^2\big] < \infty \,.
\]
\end{proposition}
\begin{proof}
See \cite[p. 226]{TAO2}.
\end{proof}

To prove the main result of this section, we need some preparation.

\begin{lemma} \label{lem:coeffi}
Let $f\in W^2(G)$, and assume that $-A_n=A_n$ for all $n$. Then, we have 
\[
c_{\chi}(f\circledast\widetilde{f})=|c_{\chi}(f)|^2 \,.
\]
\end{lemma}
\begin{proof}
Since $-A_n=A_n$, it is easy to see that $\widetilde{f}\in W^2(G)$. Then, the proof of this lemma is given in \cite[p. 227]{TAO2}. The author assumes that $f$ is weakly almost periodic, but the proof reveals that the above assumption suffices.
\end{proof}

\begin{theorem} (Parsevals equality)
Let $f\in W^2(G)$, and assume that $-A_n=A_n$ for all $n$. Then, $c_{\chi}(f)\neq0$ for at most countable many $\chi\in \widehat{G}$, and one has
\[
M\big[|f|^2\big] = \sum_{\chi\in \widehat{G}} |c_{\chi}(f)|^2 \,.
\]
\end{theorem}
\begin{proof}
Proposition~\ref{prop:conv_s_ap} tells us that $f\circledast\widetilde{f}$ is a Bohr almost periodic function. Moreover, its Fourier--Bohr coefficients are given by $|c_{\chi}(f)|^2$, see Lemma~\ref{lem:coeffi}. By Bessel's inequality (see Proposition~\ref{prop:bessel}), the series $\sum_{\chi\in \widehat{G}} |c_{\chi}(f)|^2$ converges absoultely. Hence, the Fourier--Bohr series of $f\circledast \widetilde{f}$ converges uniformly, and $f\circledast \widetilde{f}$ is its limit \cite[Lem. 4.6.11]{TAO2}. Therefore, we obtain
\[
M\big[|f|^2\big] = (f\circledast \widetilde{f})(0) = \sum_{\chi\in \widehat{G}} c_{\chi}(f\circledast\widetilde{f})\, \chi(0) = \sum_{\chi\in \widehat{G}} |c_{\chi}(f)|^2  \,. \qedhere
\]
\end{proof}

We complete the section with a uniqueness result.

\begin{theorem} \label{thm:unique}
Let $f,g\in\s(G)$ such that 
\[
c_{\chi}(f) = c_{\chi}(g) \qquad \text{ for all } \chi\in\widehat{G}. 
\]
Then, $f$ and $g$ coincide $\theta_G$-almost everywhere.
\end{theorem}

A proof of this theorem will be given in the appendix.

\section{Relations to other notions of almost periodic functions}

So far, we have only considered Stepanov and Weyl almost periodic functions. However, there are other notions of almost periodic functions. In the following, we want to see how they are connected. 

The space of Bohr almost periodic functions is denoted by $\text{SAP}(G)$.

\begin{proposition}  \label{prop:s_apcu}
One has $\s(G)\cap \Cu(G) = \text{SAP}(G)$.
\end{proposition}
\begin{proof}
$\subseteq$: Since $f$ is uniformly continuous, for every $\eps>0$, there is a neighbourhood $V$ of $0$ such that 
\begin{equation} \label{eq:2}
|f(x')-f(x'')| < \frac{\eps}{4} \,,
\end{equation}
whenever $x'-x''\in V$. 
Next, let $t\in G$ be such that 
\begin{equation} \label{eq:1}
\| f-T_tf\|_{S^p} < \left(\frac{\theta_G(V)}{\theta_G(K)}\right)^{\frac{1}{p}}\, \frac{\eps}{2}\,.
\end{equation}
Now, we claim that
\[
\|f-T_tf\|_{\infty} < \eps \,.
\]
Assume on the contrary that this statement is not true, i.e. there is an element $x_0\in G$ such that $|f(x_0)-T_tf(x_0)|>\eps$. Then, Eq.~\eqref{eq:2} implies
\[
|f(x)-T_tf(x)| > \frac{\eps}{2},
\]
as long as $x-x_0\in V$. Consequently, one has
\begin{align*}
\left(\frac{1}{\theta_G(K)}\int_{x_0+K} |f(x)-T_tf(x)|^p\ \dd \theta_G(x)
    \right)^{\frac{1}{p}}
    &\geqslant \left(\frac{1}{\theta_G(K)}\int_{x_0+V} |f(x)-T_tf(x)|^p\ \dd
      \theta_G(x) \right)^{\frac{1}{p}}  \\
    &> \frac{1}{\theta_G(K)^{\frac{1}{p}}}\, \frac{\eps}{2}\,
      \theta_G(x_0+V)^{\frac{1}{p}} \\
    &= \left(\frac{\theta_G(V)}{\theta_G(K)}\right)^{\frac{1}{p}}\, 
       \frac{\eps}{2}\,,
\end{align*}
which contradicts Eq.~\eqref{eq:1}.

\medskip

\noindent $\supseteq$: On the one hand, every weakly almost periodic function (in particular, every Bohr almost periodic function) is uniformly continuous, see \cite[Thm. 13.1]{Eb}.

On the other hand, we have
\[
\|f-T_tf\|_{S^p} = \sup_{y\in G}\left( \frac{1}{\theta_G(K)}\int_{y+K} |f(x)-T_tf(x)|^p\ \dd \theta_G(x) \right)^{\frac{1}{p}} \leqslant \|f -T_tf\|_{\infty}\,,
\]
which implies $\text{SAP}(G) \subseteq \s(G)$.
\end{proof}

\begin{remark} \label{rem:ex}
It is not possible to replace the uniform continuity by mere continuity. An example of a function which is continuous and Stepanov almost periodic but not Bohr almost periodic is given by
\[
f:\R \to \R,\quad \quad x\mapsto \sin\left( \frac{1}{2+\cos(\alpha x)+\cos(\beta x)} \right) \,,
\]
where $\alpha,\beta\in\R$ such that $\alpha,\beta$ and $\alpha\beta^{-1}$ are irrational, see \cite{lev}.
\end{remark}

\begin{definition}
A function $f \in C_{\text{u}}(G)$ is called \emph{weakly almost periodic} if the closure of $\{T_tf\ |\ t\in G\}$ is compact in the weak topology. The space of all weakly almost periodic functions is denoted by $\text{WAP}(G)$.
\end{definition}

\begin{definition}
Let $f\in \Cu(G)$ and let $\mc A=  (A_n)_{n\in\N}$ be a van Hove sequence. We define the \emph{upper absolute mean} of $f$ along $\mc A$ by
\[
\overline{M}_{\mc A}(f) := \limsup_{n\to\infty} \frac{1}{|A_n|} \int_{A_n} |f(x)|\ \dd x.
\]
Whenever the van Hove sequence $\mc A$ is clear from the context, we will write $\overline{M}(f)$ instead of $\overline{M}_{\mc A}(f)$.
\end{definition}

\begin{definition}
Fix a van Hove sequence $\mc A=(A_n)_{n\in\N}$. A function $f\in C_{\text{u}}(G)$ is called \emph{mean almost periodic} (with respect to $\mc A$) if, for all $\varepsilon>0$, the set
\[
\{t\in G\ |\ \overline{M}_{\mc A}(T_{t}f-f)<\varepsilon\}
\]
is relatively dense in $G$. The space of all mean almost periodic functions is denoted by $\text{MAP}(G)$.
\end{definition}

\begin{proposition}
We have the following chain of inclusions:
\[
\text{SAP}(G) \subseteq \se(G) \subseteq \we(G) \subseteq \text{MAP}(G) \,.
\]
\end{proposition}
\begin{proof}
The first inclusion follows from Proposition~\ref{prop:s_apcu}, the second inclusion follows from Proposition~\ref{prop:w<s}, and the third inclusion follows from Remark~\ref{rem:w_ap}.
\end{proof}

\begin{definition}
A function $f \in \text{WAP}(G)$ is called \emph{null weakly almost periodic} if $\overline{M}(f) = 0$. The set of all null weakly almost periodic functions is denoted by $\text{WAP}_{0}(G)$.
\end{definition}

\begin{proposition} \label{prop:inclusion}
We also have:
\[
\text{SAP}(G) \subseteq \text{WAP}(G)\subseteq \we(G) \subseteq \text{MAP}(G) \,.
\]
\end{proposition}
\begin{proof}
The first inclusion is trivial. 

For the second inclusion, let $f\in \text{WAP}(G)$. In that case, we can write 
\[
f=f_{\text{s}} + f_0
\] 
with $f_{\text{s}}\in\text{SAP}(G)$ and $f_0\in\text{WAP}_0(G)$, see \cite[Thm. 4.7.11]{TAO2}. Since $f_{\text{s}}\in\text{SAP}(G)$, for every $n\in\N$, there is a trigonometric polynomial $p_n$ such that
\[
\|f_{\text{s}} - p_n\|_{\infty} < \frac{1}{n} \,.
\] 
Therefore, one obtains
\[
\|f-p_n\|_{W^1} = \overline{M}(f-p_n) \leqslant \overline{M}(f_{\text{s}}-p_n) + \overline{M}(f_0) \leqslant \|f_{\text{s}}-p_n\|_{\infty} < \frac{1}{n} \,,
\]
which implies $f\in \we(G)$.

The last inclusion was proved in the previous proposition.
\end{proof}

One should note that there is no subset relation between $\text{WAP}(G)$ and $\se(G)$. The example from Remark~\ref{rem:ex} is Stepanov almost periodic but not weakly almost periodic (because it is not uniformly continuous). On the other hand, every function vanishing at infinity is weakly almost periodic but not Stepanov almost periodic (unless it is identical zero). To be more precise, one has 
\begin{align*}
\text{WAP}(G) \, \cap\, S^1(G) &= \text{SAP}(G) \,,  \\
\text{WAP}_0(G) \, \cap\, S^1(G) &= \{0\} \,. 
\end{align*}

\begin{remark}
The limit 
\[
\lim_{n\to\infty} \frac{1}{\theta_G(A_n)}\int_{A_n} f(x)\ \dd\theta_G(x)
\]
is in general not independent of $(A_n)_{n\in\N}$ if $f\in MAP(G)$. For this, consider the function
\[
h(x) =
\begin{cases}
0,& x<0,  \\
x, & 0\leqslant x \leqslant 1, \\
1, & x>1.
\end{cases}
\]
It is straightforward to check that $h\in\text{MAP}(G)$. On the other side, $h$ is not independent of $(A_n)_{n\in\N}$ because
\[
\frac{1}{n} \int_{0}^n f(x)\ \dd x =1 \qquad \text{ and } \qquad 
\frac{1}{n} \int_{-n}^{0} f(x)\ \dd x =0 \,.
\]
\end{remark}

\begin{remark}
If $f$ is a weakly almost periodic function, it is Weyl almost periodic by Proposition~\ref{prop:inclusion}. Moreover, there is a unique decomposition 
\[
f = f_{\text{s}} + f_0  \,,
\] 
where $f_{\text{s}}$ is a Bohr almost periodic function, and $f_0$ is a null weakly almost periodic function. We trivially have $\|f_0\|_{W^p}=0$, which implies
\[
\|f\|_{W^p} = \|f_{\text{s}}\|_{W^p} \,.
\]
So, the seminorm $\|\cdot\|_{W^p}$ can only see the strongly almost periodic component of a weakly almost periodic function.
\end{remark}

\appendix  
\section{}

In this last section, we will have a look at Stepanov almost periodic measures.

\begin{proposition}  \label{prop:sap=stap}
Let $\mu\in\cM^{\infty}(G)$, and let $1\leq p<\infty$. Then, $\mu\in\SAP(G)$ if and only if $\mu*f\in \s(G)$ for all $f\in \Cc(G)$.
\end{proposition}
\begin{proof}
$\Longrightarrow$: If $\mu\in\SAP(G)$, we have $\mu*f\in\text{SAP}(G)$ for all $f\in \Cc(G)$. Since every Bohr almost periodic function is stepanov almost periodic, $\mu*f\in \s(G)$ for all $f\in \Cc(G)$.

\medskip

\noindent $\Longleftarrow$: By assumption, $\mu*f\in \s(G)$ for all $f\in \Cc(G)$. Furthermore, we have $\mu*f\in\Cu(G)$, for all $f\in\Cc(G)$, because $\mu\in\cM^{\infty}(G)$. Now, Proposition~\ref{prop:s_apcu} implies the claim. 
\end{proof}

\begin{definition}
A measure $\mu\in\cM^{\infty}(G)$ is called \emph{Stepanov $p$-almost periodic} if 
\[
\mu*f\in \s(G) \qquad \text{ for all } f\in\Cc(G) \,.
\]
\end{definition}

\begin{remark} \label{rem:stap=sap}
Proposition~\ref{prop:sap=stap} says that $\mu$ is Stepanov almost periodic if and only if it is strongly almost periodic. This happens because $\mu*f\in\Cu(G)$ for every $f\in\Cc(G)$ (since $\mu$ is translation bounded) and uniformly continuous Stepanov almost periodic functions are Bohr almost periodic, see Proposition~\ref{prop:sap=stap}.  
\end{remark}

\begin{remark}
We know that a function $f\in \Cu(G)$ is Bohr almost periodic (weakly almost periodic/ null weakly almost periodic) if and only if the measure $f\,\theta_G$ is strongly almost periodic (weakly almost periodic/ null weakly almost periodic), see \cite[Prop. 4.10.5]{TAO2}. Now, if $f\in\Cu(G)$, then $f\in\s(G)$ if and only if $f\,\theta_G$ is an Stepanov $p$-almost periodic measure because
\[
f\in\Cu(G)\cap\s(G) \iff f\in\text{SAP}(G) \iff f\,\theta_G\in\SAP(G) \,,
\]
see Proposition~\ref{prop:s_apcu} and Remark~\ref{rem:stap=sap}. 
%However, if $f\in\s(G)\setminus\Cu(G)$, $f\,\theta_G$ is not a Stepanov almost periodic measure. Consider for example the function $f$ from Remark~\ref{rem:ex}. Then, $f\in\se(G)$ but $f\,\theta_G$ is not a Stepanov almost periodic measure.
\end{remark}

%We will make use of the following lemma.
%
%\begin{lemma}
%Let $f\in\s(G)$, and let $(\phi_{\alpha})_{\alpha}$ be an approximate identity. Then, one has
%\[
%\lim_{\alpha} \|f*\phi_{\alpha}-f\|_{S^p} = 0.
%\]
%\end{lemma}
%\begin{proof}
%Fix $\eps>0$ and a compact set $K\subseteq G$. By definition of an approximate identity and Lemma~\ref{lem:2}(b), there is a neighbourhood $V$ of $0$ and $\alpha_0>0$ such that
%\[
%\int_{V^c} |\phi_{\alpha}(y)|\ \dd y < \frac{\eps}{2\, \|f\|_{S^p}+1}
%\]
%for all $\alpha\geqslant \alpha_0$, and
%\[
%\sup_{y\in V}\, \|T_yf-f\|_{S^p} < \frac{\eps}{C} \,,
%\]
%where the constant $C>0$ comes from the definition of an approximate identity. This and Minkowskis triangle inequality imply
%\begin{align*}
%\|f*\phi_{\alpha}-f\|_{S^p}
%    &= \frac{1}{\theta_G(K)^{\frac{1}{p}}}\, \sup_{z\in G} \|f*\phi_{\alpha}-f\|_{L^P(z+K)}  \\
%    &\leqslant \frac{1}{\theta_G(K)^{\frac{1}{p}}}\, \sup_{z\in G} \left( \int_V + \int_{V^c} \right) 
%        \|T_yf-f\|_{L^p(z+K)}\, |\phi_{\alpha}(y)|\ \dd y   \\
%    &\leqslant C\, \sup_{y\in V} \|T_yf-f\|_{S^p} +  2\, \|f\|_{S^p}\, \int_{V^c} |\phi_{\alpha}(y)|\ \dd y  \\
%    &< \eps \,,
%\end{align*}
%which finishes the proof.
%\end{proof}

\begin{proposition} \label{prop:ac}
If $f\in\s(G)$, then  $f\, \theta_G \in \SAP(G)$.
\end{proposition}
\begin{proof} By Corollary~\ref{coro:conv_l1}, $\mu*\phi\in\s(G)$ for all $\phi\in\Cc(G)$. Consequently, $\mu$ is an Stepanov $p$-almost periodic measure. Also, it is easy to verify that $f\, \theta_G$ is translation bounded. Therefore, it is strongly almost periodic, see Remark~\ref{rem:stap=sap}. 
\end{proof}

%\begin{definition}
%Let $\mathcal{G}$ be a family of functions in $L^1(G)$. We say that $\mathcal{G}$ is \emph{equi-Stepanov almost periodic} if, for all $\eps>0$, the set 
%\[
%\big\{ t\in G\ |\  \|g-T_tg\|_{S^1} < \eps \ \text{ for all } g\in\mathcal{G} \big\}
%\]
%is relatively dense in $G$.
%\end{definition}
%
%Additionally, we will introduce the following notation:
%\[
%\mathcal{F}_K := \{ f\in L^1(G)\ |\ |f|\leqslant 1_K\} = \{ f:G\to\C \ |\ f\text{ is measurable},\, |f|\leqslant 1_K\}\,.
%\]
%
%\begin{proposition}
%Fix a compact set $K\subseteq G$. Let $\mu\in\cM^{\infty}(G)$. Then, $\mu$ is norm almost periodic only if $\mathcal{G}:= \{\mu*g\ |\ g\in \mathcal{F}_K \}$ is equi-Stepanov almost periodic.
%\end{proposition}
%\begin{proof}
%(i) Let $\mu$ be norm alomost periodic. Then, we have
%\begin{align*}
%\sup_{y\in G} \int_{y+K} |(g*\mu)(x)-T_t(g*\mu)(x)|\ \dd x
%    &= \sup_{y\in G} \int_{y+K} |(g*(\mu-T_t\mu))(x)|\ \dd x \\
%    &\leqslant \|g\|_{L^1}\, \|\mu-T_t\mu\|_K  \\
%    &\leqslant \theta_G(K)\, \|\mu-T_t\mu\|_K
%\end{align*}
%for all $g\in\mathcal{F}_K$, which implies the claim.
%\end{proof}

\bigskip

Finally, let us give a proof of Theorem~\ref{thm:unique}.

\begin{proof}
First note that, for any $h\in\s(G)$, we have $c_{\chi}(h)=c_{\chi}(h\, \theta_G)$ for all $\chi\in\widehat{G}$. Together with the assumption, this gives
\begin{equation} \label{eq:acm}
c_{\chi}(f\, \theta_G) = c_{\chi}(g\, \theta_G) \qquad \text{ for all } \chi\in\widehat{G} \,.
\end{equation}
By Proposition~\ref{prop:ac}, the measures $f\, \theta_G$ and $g\, \theta_G$ are strongly almost periodic, which implies that $f\, \theta_G-g\, \theta_G$ is strongly almost periodic, too. On the other hand, it follows from Eq.~\eqref{eq:acm} and \cite[Thm. 8.1]{ARMA} that $f\, \theta_G-g\, \theta_G$ is null weakly almost periodic. But the only measure which is strongly and null weakly almost periodic is the null measure, see \cite[Prop. 5.7]{ARMA}. Thus, we have $f\, \theta_G=g\, \theta_G$, which is equivalent to $f=g$ almost everywhere.
\end{proof}

\subsection*{Acknowledgments}  
The author wishes to thank Nicolae Strungaru for interesting discussions. The work was supported by DFG via a Forschungsstipendium with grant 415818660, and the author is grateful for the support.

\end{document}